\documentclass[12pt,reqno]{amsart}

\usepackage{amssymb}
\usepackage{verbatim}

\newcommand{\bet}{\beta}

\newcommand{\eps}{\varepsilon}
\newcommand{\gam}{\gamma}
\newcommand{\kap}{\varkappa}

\newcommand{\Ome}{\Omega}
\newcommand{\Tht}{\Theta}

\newcommand{\F}{{\mathbb F}}

\newcommand{\cL}{{\mathcal L}}
\newcommand{\cP}{{\mathcal P}}

\newcommand{\fe}{{\mathfrak e}}

\newcommand{\lpr}{\left(}
\newcommand{\rpr}{\right)}

\newcommand{\lfl}{\left\lfloor}
\newcommand{\rfl}{\right\rfloor}
\newcommand{\lcl}{\left\lceil}
\newcommand{\rcl}{\right\rceil}

\renewcommand{\>}{\rangle}

\newcommand{\stm}{\setminus}
\newcommand{\seq}{\subseteq}
\newcommand{\est}{\varnothing}

\newcommand{\longc}{,\ldots,}
\newcommand{\longe}{=\dotsb=}
\newcommand{\longp}{+\dotsb+}
\newcommand{\longle}{\le\dotsb\le}

\newcommand{\Fr}[1][r]{\F_2^{#1}}

\newcommand{\sub}[1]{{_{\substack{#1}}}}

\DeclareMathOperator{\supp}{supp}
\DeclareMathOperator{\codim}{codim}

\newtheorem{lemma}{Lemma}

\newtheorem{theorem}{Theorem}
\newtheorem{corollary}{Corollary}

\newtheorem{primetheorem}{Theorem}

\newcommand{\refl}[1]{~\ref{l:#1}}

\newcommand{\reft}[1]{~\ref{t:#1}}
\newcommand{\refc}[1]{~\ref{c:#1}}
\newcommand{\refs}[1]{~\ref{s:#1}}
\newcommand{\refb}[1]{~\cite{b:#1}}
\newcommand{\refe}[1]{~\eqref{e:#1}}

\theoremstyle{remark}
\newtheorem*{remark}{Remark}

\author{Aart Blokhuis}
\email{A.Blokhuis@tue.nl}
\address{Department of Mathematics and Computer Science,
  Eindhoven University of Technology, Netherlands}

\author{Vsevolod F. Lev}
\email{seva@math.haifa.ac.il}
\address{Department of Mathematics, The University of Haifa at Oranim,
  Israel}

\title{Flat-containing and shift-blocking sets in $\Fr$}

\begin{document}
\baselineskip=16pt

\begin{abstract}
For non-negative integers $r\ge d$, how small can a subset $C\seq\Fr$ be,
given that for any $v\in\Fr$ there is a $d$-flat passing through $v$ and
contained in $C\cup\{v\}$? Equivalently, how large can a subset $B\seq\Fr$
be, given that for any $v\in\Fr$ there is a linear $d$-subspace not blocked
non-trivially by the translate $B+v$? A number of lower and upper bounds are
obtained.
\end{abstract}

\maketitle

\section{Introduction}\label{s:intro}

The well-known finite-field version of the Kakeya problem is to estimate the
smallest size of a subset of a finite vector space, containing a line in
every direction. A natural dual problem is to estimate the smallest size of a
subset containing a line \emph{through every point} of the space, with the
possible exception of the point itself. (The problem would become trivial had
we not excluded the anchor point from consideration. This can be considered
as an analogue of forbidding the zero difference in the definition of a
progression-free set.) More generally, given an integer $d$ one can consider
sets ``essentially containing'' a \emph{$d$-flat} through every point of the
space. This motivates the following definitions.

Let $V$ be a finite vector space and $d$ an integer with $0\le d\le\dim V$.
We say that a subset $C\seq V$ is \emph{$d$-complete} if for every $v\in V$
there exists a $d$-subspace $L_v\le V$ such that
  $$ v+(L_v\stm\{0\}) \seq C; $$
that is, through every point of $V$ passes a $d$-flat entirely contained in
$C$, save, perhaps, for the point itself. Equivalently, $C\seq V$ is
$d$-complete if any translate of $C$ contains all non-zero vectors of some
$d$-subspace. By $\gam_V(d)$ we denote the smallest possible size of a
$d$-complete subset $C\seq V$; alternatively, $\gam_V(d)$ is the smallest
possible size of a union of the form
  $$ \textstyle \bigcup_{v\in V} \big(v+(L_v\stm\{0\})\big), $$
for all families $\{L_v\colon v\in V\}$ of $d$-subspaces.

Clearly, a subset $C\seq V$ is $d$-complete if and only if its complement
$B:=V\stm C$ has the property that for every $v\in V$ there is a $d$-subspace
$L_v\le V$ with
  $$ \big(v+(L_v\stm\{0\})\big)\cap B = \est. $$
We call sets with this property \emph{non-blocking}; quantitatively speaking,
$B\seq V$ is $d$-non-blocking if through every point of $V$ passes a
co-$d$-flat disjoint with $B$, with the possible exception of the point
itself. Equivalently, $B\seq V$ is $d$-non-blocking if for any translate of
$B$ there is a co-$d$-subspace of $V$, avoiding all non-zero points of the
translate. We denote by $\bet_V(d)$ the largest possible size of a
$d$-non-blocking subset $B\seq V$.

The significance of the quantity $\bet_V(d)$ lies in the fact that every
subset $B\seq V$ of size $|B|>\bet_V(d)$ is guaranteed to have a translate
blocking (that is, having non-zero and non-empty intersection with) all
co-$d$-subspaces of $V$.

Writing $r:=\dim V$, from the discussion above we have
\begin{equation}\label{e:bet-gam}
  \bet_V(d) = |V|-\gam_V(r-d)
\end{equation}
for all $0\le d\le r$. In view of this basic relation, our results can be
equivalently stated in terms of either of the quantities $\gam_V$ and
$\bet_V$. We do not follow any strong rule in this respect, each time
choosing whatever seems more natural to us. In some cases, two restatements
are given; in most other cases the result is stated in terms of $\gam_V$ if
it is of primary interest for flats of low dimension, and in terms of
$\bet_V$ when it is mostly interesting for flats of low \emph{co}-dimension.

It is straightforward to verify that for every finite vector space $V$ of
dimension $r:=\dim V\ge 1$ we have
\begin{equation}\label{e:mono-gamma}
  0 = \gam_V(0) < \gam_V(1) \longle \gam_V(r-1) < \gam_V(r) = |V|;
\end{equation}
equivalently,
\begin{equation}\label{e:mono-beta}
  0 = \bet_V(0) < \bet_V(1) \longle \bet_V(r-1) < \bet_V(r) = |V|.
\end{equation}

In what follows we confine ourselves to the situation where $V$ is a vector
space over the two-element field $\F_2$. We denote the $r$-dimensional vector
space over this field by $\Fr$, and we abbreviate $\gam_{\Fr}(d)$ as
$\gam_r(d)$, and $\bet_{\Fr}(d)$ as $\bet_r(d)$.

We present our results in three blocks. In Section~\refs{basic} we make some
basic observations and in particular, find $\gam_r(1)$ and $\bet_r(1)$ (hence
also $\gam_r(r-1)$ and $\bet_r(r-1)$, cf.~\refe{bet-gam}), and determine
$\gam_r(2)$ up to a multiplicative factor; the very short proofs are also
included in Section~\refs{basic}. Non-existence results showing that complete
sets are large (and accordingly, non-blocking sets are small) are presented
in Section~\refs{summaryI}. Section~\refs{summaryII} lists a number of
upper-bound estimates for $\gam_r$ (hence, lower-bound estimates for
$\bet_r$), based on specific constructions of complete and non-blocking sets.

The proofs of the results discussed in Sections~\refs{summaryI} and
\refs{summaryII} are given in Sections~\refs{proofsI} and \refs{proofsII},
respectively.

\section{Motivation and acknowledgement}

Our initial motivation came from the following problem raised by Ernie Croot
(personal communication with the second-named author). Suppose that to each
$v\in\Fr$ there corresponds a subset $A_v\seq\Fr$, and write $2\cdot A_v$ for
the set of all non-zero elements of $\Fr$, representable as a sum of two
elements of $A_v$. Given that all sets $A_v$ are large, how large must the
union $C:=\cup_{v\in\Fr}(v+2\cdot A_v)$ be? Does there exist a constant $c>1$
such that if $|A_v|>2^r/r^c$, then $C$ contains all but at most $2^r/r^c$
elements of $\Fr$? In the special case where all $A_v$ are actually affine
subspaces of $\Fr$, this question can be restated in our present terms: is it
true that if $d<c\log_2 r$, then $\bet_r(d)<2^r/r^c$? The reader will easily
check that Theorem \reft{beta-upper} below yields much stronger estimates:
say, we have $\bet_r(d)<2^{0.85r}$ whenever $d<0.15r$. However, the general
case where $A_v$ are arbitrary sets (not necessarily affine subspaces) cannot
be treated with our present approach.

We are grateful to Ernie Croot for bringing this problem to our attention.

\section{Basic observations: lines, hyperplanes, and $2$-flats}\label{s:basic}

The quantities $\gam_r(1)$ and $\bet_r(1)$ and, consequently, $\gam_r(r-1)$
and $\bet_r(r-1)$, are easy to determine.

\begin{theorem}\label{t:gamma1}
For every integer $r\ge 1$ we have $\gam_r(1)=\bet_r(1)=2$. Hence,
$\bet_r(r-1)=\gam_r(r-1)=2^r-2$.
\end{theorem}

\begin{proof}
The equality $\gam_r(1)=2$ follows from the observation that a singleton set
does not contain a $1$-flat passing through its unique element, whereas for
any two-element set $C\seq\Fr$ and any $v\in\Fr$, there is $1$-flat passing
through $v$ and contained in $C\cup\{v\}$.

To find $\bet_r(1)$ we first notice that for any two-element set $B\seq\Fr$
there is a linear co-$1$-subspace, disjoint from $B\stm\{0\}$; hence
$\bet_r(1)\ge 2$. On the other hand, if $B=\{b_1,b_2,b_3\}\seq\Fr$ is a
three-element set, then the translate
$(b_1+b_2+b_3)+B=\{b_1+b_2,b_2+b_3,b_3+b_1\}$ blocks every linear
co-$1$-subspace: for, the vectors $b_1+b_2,\,b_2+b_3$, and $b_3+b_1$ add up
to $0$, and therefore they are not simultaneously contained in the complement
of a linear co-$1$-subspace. Thus, $\bet_r(1)\le 2$, and it follows that,
indeed, $\bet_r(1)=2$.
\end{proof}

To estimate $\gam_r(2)$ we remark that if every element of $\Fr$ is a sum of
three pairwise distinct elements of a set $S\seq\Fr$, then $\binom{|S|}3\ge
2^r$, whence $|S|>\sqrt[3]6\cdot2^{r/3}$. On the other hand, sets $S$ of size
$|S|=O(2^{r/3})$, with the property just mentioned, are known to exist: see,
for instance, \cite[Theorem 5.4.28]{b:chll}, or consider a decomposition of
$\Fr$ into the direct sum of three subspaces of roughly equal dimension and
take $S$ to be their union.
\begin{theorem}\label{t:gamma2}
If $r\ge 2$, then $\gam_r(2)$ is the smallest cardinality of a subset
$C\seq\Fr$ with the property that every element of $\Fr$ is representable as
a sum of three pairwise distinct elements from $C$. Consequently,
$\gam_r(2)=\Tht\big(2^{r/3}\big)$.
\end{theorem}

\begin{proof}
Just notice that for given vectors $v,c_1,c_2,c_3\in\Fr$ to form a $2$-flat
it is necessary and sufficient that $v=c_1+c_2+c_3$ and $c_1,c_2,c_3$ are
pairwise distinct.
\end{proof}

An interesting property of the quantity $\gam_r(d)$ is that for any fixed
value of $d$, it is sub-multiplicative in $r$.
\begin{lemma}\label{l:sub-mult}
For any integer $r_1,r_2\ge d\ge 0$ we have
  $$ \gam_{r_1+r_2}(d) \le \gam_{r_1}(d) \gam_{r_2}(d). $$
\end{lemma}

\begin{proof}
Let $r:=r_1+r_2$ and write $\Fr=V_1\oplus V_2$, where $\dim V_1=r_1$ and
$\dim V_2=r_2$. Find $C_1\seq V_1$ and $C_2\seq V_2$ such that for
$i\in\{1,2\}$ we have $|C_i|=\gam_{r_i}(d)$ and $C_i$ is $d$-complete in
$V_i$. We claim that $C_1+C_2$ is $d$-complete in $\Fr$, so that
  $$ \gam_{r_1+r_2}(d)\le |C_1+C_2|=|C_1||C_2|=\gam_{r_1}(d)\gam_{r_2}(d). $$
To see this, fix $v\in\Fr$, write $v=v_1+v_2$ with $v_1\in V_1$ and
 $v_2\in V_2$, find $d$-flats $F_1\seq V_1$ and $F_2\seq V_2$ such that
$v_i\in F_i\seq C_i\cup\{v_i\}$ for $i\in\{1,2\}$, and select arbitrarily
bases $\{e_1\longc e_d\}$ and $\{f_1\longc f_d\}$ of the linear $d$-spaces
$v_1+F_1$ and $v_2+F_2$, respectively. Then all points of the $d$-flat
$v+\<e_1+f_1\longc e_d+f_d\>$, other than $v$, are contained in $C_1+C_2$.
\end{proof}

Using a standard argument, it is easy to derive from Lemma \refl{sub-mult}
that to any fixed $d\ge 1$ there corresponds some $\kap_d\in[0,1]$ such that
$\gam_r(d)=2^{(\kap_d+o(1))r}$ as $r\to\infty$. As it follows from Theorems
\reft{gamma1} and \reft{gamma2}, we have $\kap_1=0$ and $\kap_2=1/3$. For
$d\ge 3$ the precise values of $\kap_d$ are not known to us, but we will see
that $3/8\le \kap_3\le 3/7$ (Theorems~\reft{gamma-lower}
and~\reft{gamma-upper}), and that $\kap_d<1/2$ for all $d$
(Theorem~\reft{gamma-upper}).

\section{Non-existence results:
  lower bounds for $\gam_r$, upper bounds for $\bet_r$}\label{s:summaryI}

By Theorem \reft{gamma2} and \refe{mono-gamma}, we have
$\gam_r(3)=\Ome(2^{r/3})$. The following theorem presents an improvement of
this estimate.
\begin{theorem}\label{t:gamma-lower}
If $r\ge 15$ is an integer, then
  $$ \gam_r(3) > c\cdot 2^{3r/8}, $$
where $c=(16464)^{1/8}\approx 3.3656$.
\end{theorem}

The argument employed in the proof of Theorem \reft{gamma-lower} (see Section
\refs{proofsI}) can also be used to estimate $\gam_r(3)$ non-trivially for
$3\le r\le 14$; say, it is easy to derive from \refe{gamma-lower1} that
$\gam_r(3)>2^{r/2}$ for every such $r$. The only reason to confine to
 $r\ge 15$ is that this allows us to keep the coefficient
$c$ reasonably large.

\begin{corollary}\label{c:gamma3lower}
For integer $r\ge d\ge 3$, we have $\gam_r(d)=\Ome(2^{3r/8})$ with an
absolute implicit constant.
\end{corollary}

Recall, that the entropy function is defined by
  $$ H(x) := -x\ln x-(1-x)\ln(1-x),\ 0<x<1, $$
and that
\begin{equation}\label{e:binomial}
  \frac1{\sqrt{2r}} \: e^{rH(d/r)} \le \binom rd
                                < \sum_{j=0}^d \binom rj \le e^{rH(d/r)}
\end{equation}
for all integer $1\le d\le r/2$; this follows easily, for instance, from
\cite[Ch. 10,~\S 11, Lemmas~7 and~8]{b:mcwsl}. (Although this is not used
below, we remark that the expression in the right-hand side of
\refe{binomial} can be given a nice symmetrical form; namely,
$e^{rH(d/r)}=r^r/(d^d(r-d)^{r-d})$.)

Using \refe{binomial}, it is easy to verify that our next theorem improves
Corollary \refc{gamma3lower} for flats of dimension $d\gtrsim 0.073r$ (by
which we mean $d>(\kap+o(1))r$ with an absolute constant
 $\kap\approx 0.073$).

\begin{theorem}\label{t:beta-upper}
For integer $r\ge d\ge 0$ we have
\begin{equation}\label{e:gamma-binomial}
  \gam_r(d) \ge \sum_{j=0}^{d-1} \binom rj.
\end{equation}
Equivalently,
\begin{equation}\label{e:beta-binomial}
  \bet_r(d) \le \sum_{j=0}^d \binom rj.
\end{equation}
\end{theorem}

In Section \refs{proofsI}, two proofs of Theorem \reft{beta-upper} are given.
Elaborating on one of them, we will also establish the following slight
refinement.
\begin{primetheorem}\label{t:beta-upper-strong}
For integer $r\ge d\ge 0$ we have
  $$ (1-2^{d-r})\bet_r(d) \le \sum_{j=0}^d \binom rj - 2^d. $$
\end{primetheorem}

It is not difficult to derive from Theorem \reft{beta-upper-strong}, for
instance, that $\bet_r(d)\le\sum_{j=0}^d\binom{r}{j}-2^{d-1}$ whenever
$d<r/2$, or that $\bet_r(d)\le\sum_{j=0}^d\binom{r}{j}-2^{d}$ whenever
$d<0.227r$ (for the latter conclusion assume, for a contradiction, that
$\bet_r(d)\ge\sum_{j=0}^d\binom{r}{j}-2^{d}+1$, and use \refe{binomial}).

\section{Constructions: upper bounds for $\gam_r$, lower bounds for $\bet_r$}
  \label{s:summaryII}

All bounds listed in this section are constructive. We confine here to the
resulting estimates and comparison between them, with the underpinning
constructions being incorporated into the proofs (presented in Section
\refs{proofsII}). For the background material in coding theory (simplex
codes, dual-BCH codes, Greismer and Carlitz-Uchiyama bounds), the reader can
refer any standard textbook, such as \cite{b:mcwsl,b:vl}.

\begin{theorem}\label{t:gamma-upper}
For any integer $r\ge d\ge 3$ we have
  $$ \gam_r(d) < K_d\cdot 2^{\lpr\frac12-\eps_d\rpr r}, $$
where $\eps_d=\frac1{2(2^d-1)}$ and $K_d=(2^d-1)2^{2^{d-1}-(3/2)+\eps_d}$.
\end{theorem}

As a particular case to be compared against Theorem \reft{gamma-lower}, we
have $\gam_r(3)=O(2^{3r/7})$.

As the reader will see, the proof of Theorem \reft{gamma-upper} relies on the
properties of simplex codes. The reason to prefer simplex codes over other
codes is that these codes have the largest possible relative minimum distance
among all codes of given dimension $d$ (as the Griesmer bound readily shows).
The drawback of the simplex codes, on the other hand, is that their length is
exponential in the dimension, leading eventually to the double-exponential
dependence on $d$ in the constant $K_d$ of Theorem \reft{gamma-upper}, and
hence resulting in very poor bounds as $d$ grows. Indeed, the estimate of the
theorem becomes trivial for $d\sim\log r$. Using other codes one can produce
non-trivial estimates for reasonably large values of $d$. Specifically, the
argument employed in the proof of Theorem \reft{gamma-upper} shows that if
$n$ and $\mu$ are positive integers such that there exists a code $S$ of
length $n$, minimum distance $\mu$, dimension $d$, and the largest weight $M$
satisfying $(n-M)\lfl r/n\rfl\ge d$, then
  $$ \gam_r(d) < 2^{(1-\mu/n)r+n+d-\mu}. $$
Indeed, it suffices that the dimension of $S$ be \emph{at least} $d$, as it
follows by considering any subcode of $S$ of dimension $d$. Choosing $S$ to
be the dual of a BCH code with appropriately chosen parameters, we prove
\begin{theorem}\label{t:BCH}
There exists an absolute constant $K$ such that for any integer $r\ge d\ge3$
we have
  $$ \gam_r(d) < 2^{0.5r + K(dr/\log_2r)^{2/3}}. $$
Hence, if $d=o(\sqrt r\log_2 r)$, then $\gam_r(d)<2^{(0.5+o(1))r}$.
\end{theorem}

We now turn to estimates which (unlike those of Theorems \reft{gamma-upper}
and \reft{BCH}) are mostly of interest for flats of low \emph{co}-dimension.

\begin{theorem}\label{t:beta-lower}
For integer $r$ and $d$ with $2\le d\le r/2$, let $\rho$ denote the remainder
of the division of $r$ by $2d$. Then
  $$ \bet_r(d) = \sum_\sub{0\le i\le 2d-\rho,\,0\le j\le\rho \\ i+j=d}
          \binom{2d-\rho}{i} \binom{\rho}{j}
                 \lfl\frac{r}{2d}\rfl^i \lpr\lfl\frac{r}{2d}\rfl+1\rpr^j $$

Consequently,
  $$ \bet_r(d) \ge \binom{2d}{d}\lfl\frac{r}{2d}\rfl^d. $$
\end{theorem}

We notice that Theorems \reft{beta-upper} and \reft{beta-lower} give
$\bet_r(d)=\Ome_d(r^d)$. It follows, say, that $\gam_r(r-2)=2^r-\Ome(r^2)$;
compared with Theorem~\reft{gamma2} this shows that $\gam_r$, considered as a
function of $d\in[0,r]$, exhibits a highly asymmetric behavior.

\begin{theorem}\label{t:beta-lowerprime}
For integer $r\ge d\ge 2$, let $\rho$ denote the remainder of the division of
$r$ by $d$. Then
  $$ \bet_r(d) \ge \lpr \lfl\frac{r}{d}\rfl+1 \rpr^{d-\rho}
                                    \lpr \lfl\frac{r}{d}\rfl+2\rpr^{\rho}. $$
Consequently,
  $$ \bet_r(d)>(r/d)^d, $$
and if $d\ge r/2$, then
  $$ \bet_r(d) \ge \lpr\frac32\rpr^r \lpr\frac43\rpr^d. $$
\end{theorem}

While the bounds of Theorems \reft{beta-lower} and \reft{beta-lowerprime} may
not be easy to compare analytically, computations suggest that
Theorem~\reft{beta-lower} gives a better estimate for all $d\le r/2$, save
for a finite (and small) number of exceptional pairs $(d,r)$. For $d>r/2$
Theorem \reft{beta-lower} yields
\begin{equation}\label{e:d>r/2}
  \bet_r(d) \ge \bet_r(\lfl r/2)\rfl \ge \binom{2\lfl r/2\rfl}{\lfl r/2\rfl},
\end{equation}
which is superseded by Theorem \reft{beta-lowerprime} for $d$ very close to
$r$; namely, for $r-d\lesssim 1.738\ln r$. We also notice that if, indeed,
$r-d\lesssim 1.443\ln r$, then Theorem \reft{beta-lowerprime} itself is
superseded by Theorem \reft{gamma-upper}.

\begin{theorem}\label{t:beta-lower2}
Suppose that $r\ge d\ge 1,\ k\ge 1$, and $r_i\ge d_i\ge 0\ (i=1\longc k)$ are
integers such that $r_1\longp r_k\le r,\ d_1\longp d_k\le d$, and
 $r_i\le d+d_i$ for $i=1\longc k$. Then
  $$ \bet_r(d) \ge \binom{r_1}{d_1}\dotsb\binom{r_k}{d_k}. $$
\end{theorem}

It is not difficult to see that for $d\ge r/2$, Theorem \reft{beta-lower2}
gives
  $$ \bet_r(d) \ge \binom{r}{\lfl r/2\rfl}; $$
this is identical or marginally stronger than \refe{d>r/2}. For $1\le d\le
r/2$, the maximum of the product $\binom{r_1}{d_1}\dotsb\binom{r_k}{d_k}$
under the constrains $r_i\ge d_i\ge 0,\ r_1\longp r_k\le r$, and
 $d_1\longp d_k\le d$ (with $r_i\le d+d_i$ \emph{not} assumed!) is
$\binom{r}{d}$; this is to be compared with Theorem \reft{beta-upper} and
also with the following corollary.
\begin{corollary}\label{c:rk}
If $r\ge 1$ and $\sqrt r<d\le r/2$ are integer, then
  $$ \bet_r(d) > e^{rH(d/r) - 2(r/d) \ln r}. $$
Consequently, if $d/\sqrt{r}\to\infty$ and $d\le r/2$, then
  $$ \bet_r(d) > \binom{r}{d}^{1+o(1)}. $$
\end{corollary}

A precise comparison between Theorems \reft{beta-lower} and
\reft{beta-lower2} is hardly feasible. However, the ``main terms''
(cf.~Corollary~\refc{rk}) are easy to compare, and it turns out that
  $$ e^{rH(d/r)} > \binom{2d}d \lpr\frac{r}{2d}\rpr^d $$
for all positive integer $r$ and $d\le r/2$. We remark, on the other hand,
that Theorem \reft{beta-lower2} fails to produce reasonable bounds if $d$ is
very small (as compared to $r$).

\section{Proofs, I: non-existence results}\label{s:proofsI}

\begin{proof}[Proof of Theorem \reft{gamma-lower}]
Suppose that $C\seq\Fr$ is $3$-complete; that is, every element $v\in\Fr$
lies in a $3$-flat $F_v$ with the other seven elements in $C$. We want to
show that if $r\ge 15$, then $|C|>c\cdot2^{3r/8}$.

Let $S$ be the set of all those $s\in\Fr$ representable as a sum of two
distinct elements of $C$, and for each $s\in S$ denote by $\nu(s)$ the number
of such representations, with two representations that differ by the order of
summands considered identical. Write $B:=\Fr\stm C$, and for each $s\in S$
let $B(s)$ be the set of all those $b\in B$ with $s\in b+F_b$. (Notice, that
$b+F_b$ is the linear $3$-subspace, parallel to the flat $F_b$.) Thus, every
$b\in B$ belongs to exactly seven sets $B(s)$, and hence
\begin{equation}\label{e:allWs}
  \sum_{s\in S} |B(s)| = 7|B|.
\end{equation}

For every $s\in S$ and $b\in B(s)$ there are three distinct representations
$s=c_1+c_2$ with $c_1,c_2\in F_b\cap C$. Since these representations uniquely
determine $F_b$, and hence $b$ itself, we have
\begin{equation}\label{e:choose3}
  |B(s)| \le \binom{\nu(s)}{3} < \frac 16\, (\nu(s))^3.
\end{equation}
Also,
\begin{equation}\label{e:|A|}
  |B(s)| \le |C|
\end{equation}
as $b+s\in F_b\stm\{b\}$ for each $b\in B(s)$, implying $B(s)+s\seq C$.

Averaging multiplicatively \refe{choose3} and \refe{|A|} with the exponent
weights $1/3$ and $2/3$, respectively, we get
  $$ |B(s)| < \frac1{\sqrt[3]6}\,|C|^{2/3} \nu(s), $$
and substitution into \refe{allWs} yields
\begin{equation}\label{e:gamma-lower1}
  7(2^r-|C|) < \frac1{\sqrt[3]6}
                    \: |C|^{2/3} \sum_{s\in S} \nu(s)
                           = \frac1{\sqrt[3]6} \: |C|^{2/3} \binom{|C|}2.
\end{equation}
If $|C|\ge 2^{r/2}$, then we are done as $2^{r/2}>c\cdot 2^{3r/8}$ for
 $r\ge 15$. If $|C|<2^{r/2}$, then $(2^r-|C|)/(|C|-1)>2^r/|C|$; consequently,
\refe{gamma-lower1} gives
  $$ |C|^{8/3} > 14\sqrt[3]6 \cdot 2^r $$
implying the result.
\end{proof}

Next, we give two proofs of Theorem \reft{beta-upper}. Both proofs rely on
the fact that if $\cL_{r,d}$ is the vector space of all multilinear
polynomials in $r$ variables over the field $\F_2$ of total degree at most
$d$, then $\dim\cL_{r,d}=\sum_{j=0}^d\binom rj$ (which is immediate from
looking at the ``monomial basis''). Nevertheless, the two proofs seem to
differ significantly. We keep using the notation $\cL_{r,d}$ below.

Our first proof goes along the lines of Dvir's proof \refb{d} of the finite
field Kakeya conjecture. We need two basic facts about polynomials over the
field $\F_2$.

\medskip
\noindent
{\bf Fact 1.\,} For integer $d\ge 1$, a polynomial in $d$ variables over
$\F_2$ of degree smaller than $d$ cannot vanish on all, but at most one point
of $\Fr[d]$. (To see this, observe that every monomial of the polynomial in
question is independent of at least one variable, hence the sum of its values
over $\Fr[d]$ is equal to $0$.)

\medskip
\noindent
{\bf Fact 2.\,} For integer $r\ge 1$, a non-zero multilinear polynomial in
$r$ variables over $\F_2$ cannot vanish on all points of $\Fr$. (For the
proof, notice that every function from $\Fr$ to $\F_2$ can be represented by
a multilinear polynomial, and that both the total number of all functions and
the total number of all multilinear polynomials are equal to $2^{2^r}$. Thus,
every function is \emph{uniquely} represented by such a polynomial.)

\begin{proof}[First proof of Theorem \reft{beta-upper}]
Assuming that \refe{gamma-binomial} is false, find a $d$-complete set
$C\seq\Fr$ with $|C|<\sum_{j=0}^{d-1} \binom rj$. Thus, for every $v\in\Fr$
there exists a $d$-subspace $L_v\le\Fr$ with $v+(L_v\stm\{0\})\seq C$. Since
$\dim\cL_{r,d-1}=\sum_{j=0}^{d-1} \binom rj$, the evaluation map from
$\cL_{r,d-1}$ to $\F_2^{|C|}$ (sending every polynomial to the $|C|$-tuple of
its values at the points of $C$) is degenerate. Hence, there is a non-zero
polynomial $P\in\cL_{r,d-1}$ vanishing at every point of $C$. As a result,
for each $v\in\Fr$ there exist $v_1\longc v_d\in\Fr$ such that
  $$ P(v+t_1v_1\longp t_dv_d)=0,\quad (t_1\longc t_d)\in\F_2^d\stm\{0\}. $$
This means that the polynomial
  $$ P(v+T_1v_1\longp T_dv_d)\in\F_2[T_1\longc T_d] $$
vanishes at every point of $\F_2^d\stm\{0\}$. The degree of this polynomial
is at most $\deg P\le d-1$. Hence, by Fact~1, we have $P(v)=0$; that is, $P$
vanishes at every point of $\Fr$. This, however, contradicts Fact~2.
\end{proof}

\begin{proof}[Second proof of Theorem \reft{beta-upper}]
Aiming at \refe{beta-binomial}, fix $B\seq\Fr$ with $|B|=\bet_r(d)$ such that
to every $b\in B$ there corresponds a co-$d$-flat $F_b\seq\Fr$ with $F_b\cap
B=\{b\}$. For every such flat, find a polynomial $P_b\in\cL_{r,d}$ with
$P_b(z)=1$ whenever $z\in F_b$, and $P_b(z)=0$ otherwise. (Such a polynomial
can be constructed by taking the product of $d$ linear factors corresponding
to $d$ hyperplanes whose intersection is $F_b$.) These $|B|$ polynomials are
linearly independent, as it follows by substituting the points $b\in B$ into
their linear combinations. Consequently,
  $$ \bet_r(d)=|B|\le\dim\cL_{r,d}=\sum_{j=0}^d\binom rj. $$
\end{proof}

\begin{proof}[Proof of Theorem \reft{beta-upper}$'$]
We elaborate on the second proof of Theorem \reft{beta-upper}. Fix a
$d$-non-blocking set $B\seq\Fr$ with $|B|=\bet_r(d)$, and write
 $C:=\Fr\stm B$. For each $v\in\Fr$ find a co-$d$-flat $F_v$ with
$v\in F_v\seq C\cup\{v\}$ and, as above, let $P_v\in\cL_{r,d}$ be an
``indicator polynomial'' of $F_v$. Write $\cP_B:=\{P_b\colon b\in B\}$ and
$\cP_C:=\{P_c\colon c\in C\}$. Notice, that the subspace of $\cL_{r,d}$
generated by $\cP_B$ intersects trivially the subspace generated by $\cP_C$:
for if
  $$ \sum_{b\in B} \eps(b) P_b = \sum_{c\in C} \eps(c) P_c $$
with $\eps\colon\Fr\to\F_2$ then, evaluating at any specific $b\in B$, we get
$\eps(b)=0$. Thus, denoting by $L_C$ the subspace, generated by $\cP_C$, we
have
  $$ |B| \le \dim\cL_{r,d} - \dim L_C, $$
and we claim that $\dim L_C\ge2^{-(r-d)}|C|$. To see this, we observe
that for any subset $C_0\seq C$ with $|C_0|<2^{-(r-d)}|C|$, we have
  $$ \textstyle \left|\bigcup_{c\in C_0} F_c\right| \le 2^{r-d}
                                                       \cdot |C_0| < |C|, $$
and that for any $c'\notin\cup_{c\in C_0} F_c$, the polynomial $P_{c'}$
is not a linear combination of the polynomials $P_c$ with $c\in C_0$.
Consequently, we have
  $$ |B| \le \dim \cL_{r,d}-2^{-(r-d)}|C|
                                   = \dim \cL_{r,d}-2^{-(r-d)} (2^r-|B|), $$
whence
  $$ (1-2^{d-r})|B| \le \sum_{j=0}^d \binom rj - 2^d $$
implying the result.
\end{proof}

\section{Proofs, II: constructions}\label{s:proofsII}

\begin{proof}[Proof of Theorem \reft{gamma-upper}]
For $r<2^d-1$ the assertion follows, by a straightforward computation, from
the trivial estimate $\gam_r(d)<2^r$. Suppose, therefore, that $r\ge 2^d-1$.

Write $n:=2^d-1$ and let $S<\F_2^n$ be the simplex code of length $n$. Thus,
$S$ is a $d$-subspace of $\F_2^n$, generated by the rows of the $d\times n$
matrix whose columns are all the non-zero vectors in $\F_2^d$, and every
non-zero element of $S$ has weight $2^{d-1}$ with respect to the standard
basis of $\F_2^n$. Choose subspaces $V_1\longc V_n\le\Fr$ so that
 $\Fr=V_1\oplus\dotsb\oplus V_n$ and the dimension of each $V_i$ is either
$\lfl r/n\rfl$ or $\lcl r/n\rcl$, and consider the set
  $$ \textstyle
     C := \bigcup_{(s_1\longc s_n)\in S\stm\{0\}} \quad %
                                  \bigoplus_{i\in[1,n] \colon s_i=0} V_i. $$
We have
\begin{align*}
  |C| &<   \sum_{c\in S\stm\{0\}} 2^{(n-2^{d-1})\lcl r/n\rcl} \\
      &\le (2^d-1)\, 2^{(2^{d-1}-1) \, \lpr \frac{r-1}{n}+1\rpr} \\
      &=   K_d\cdot 2^{\lpr\frac12-\eps_d\rpr r},
\end{align*}
and to complete the proof we show that for every $v\in\Fr$ there is a
$d$-flat passing through $v$ and contained in $C\cup\{v\}$. To this end, we
write $v=v_1\longp v_n$ with $v_i\in V_i$ for $=1\longc n$, and let
\begin{align*}
  F_v
    &:= \Big\{ {\sum}_{i\in[1,n]\colon s_i=0} v_i \colon
                                          (s_1\longc s_n)\in S \Big\} \\
    &\:=  v + \{ s_1v_1\longp s_nv_n \colon (s_1\longc s_n)\in S \}.
\end{align*}
Evidently, $F_v$ is a flat with $v\in F_v\seq C\cup\{v\}$. Moreover, the
dimension of $F_v$ is at most $d$. If it is equal to $d$, then we are done.
Otherwise, there exists an element $(s_1\longc s_n)\in S\stm\{0\}$ such that
$s_1v_1\longp s_nv_n=0$; equivalently, $v$ is an element of the subspace
$\oplus_{i\in[1,n]\colon s_i=0} V_i\seq C$, and we conclude the proof
observing that the dimension of this subspace is at least
 $(2^{d-1}-1)\lfl r/n\rfl\ge d$.
\end{proof}

As we have mentioned in Section \refs{summaryII} (and the reader can easily
check now), the argument employed in the proof of Theorem \reft{gamma-upper}
shows that if $n$ and $\mu$ are positive integers such that there exists a
code $S$ of length $n$, minimum distance $\mu$, dimension at least $d$, and
the largest weight $M$ satisfying $(n-M)\lfl r/n\rfl\ge d$, then
\begin{equation}\label{e:gamma-upper-generic}
  \gam_r(d) < 2^{(1-\mu/n)r+n+d-\mu}.
\end{equation}
This observation is used in the proof of Theorem \reft{BCH} below.

\begin{proof}[Proof of Theorem \reft{BCH}]
Consider the code dual to the BCH code with the parameters $m$ and $e$
defined by
  $$ m := \lcl \frac 23\,\big(\log_2(dr)-\log_2\log_2(dr)\big)\rcl,
                                                \ e := \lcl \frac dm\rcl. $$
This is a code of length $n:=2^m-1$, with the weight of every non-zero code
word in the interval $[0.5n-(e-1)\sqrt n,0.5n+(e-1)\sqrt n\,]$ and
consequently, having the minimum distance
  $$ \mu \ge 0.5n-(e-1)\sqrt n $$
and the maximum distance
  $$ M \le 0.5n+(e-1)\sqrt n $$
(the Carlitz-Uchiyama bound).

We notice that $r\ge d\ge 3$ implies $rd\ge 9$, whence
\begin{equation}\label{e:BCH-mlarge}
  m \ge 0.25 \log_2(dr).
\end{equation}
Also,
\begin{equation}\label{e:BCH-n}
  \frac12\,\left( \frac{dr}{\log_2(dr)} \right)^{2/3}
                \le n < 2\left( \frac{dr}{\log_2(dr)} \right)^{2/3}
\end{equation}
(the first inequality following from $n=2^m-1\ge2^{m-1}$).

Assuming now
\begin{equation}\label{e:BCH-d}
  d < c \sqrt r\log_2 r
\end{equation}
with a sufficiently small absolute constant $c>0$ (as we clearly can,
choosing $K$ large enough), by \refe{BCH-n} we get
\begin{equation}\label{e:BCH-r}
  r > c^{-2/3} \left( \frac{dr}{\log_2(dr)} \right)^{2/3} > n
\end{equation}
and, by \refe{BCH-n}, \refe{BCH-d}, and \refe{BCH-mlarge},
  $$ \frac{e-1}{\sqrt n} < \frac{d}{m\sqrt n}
      < 2 \frac dm\left( \frac{\log_2(dr)}{dr}\right)^{1/3}
      < 2c^{2/3} \, \frac{\log_2(dr)}{m} < 0.25. $$
As a result, and taking into account \refe{BCH-r} and \refe{BCH-d},
  $$ (n-M)\lfl r/n\rfl >  (0.5n-(e-1)\sqrt n) \frac r{2n} > 0.125 r > d. $$
Furthermore, a straightforward computation confirms that \refe{BCH-d} yields
$e\le 2^{\lcl m/2\rcl-1}$, which is known to imply that the dimension of the
code under consideration is $em\ge d$. The result now follows by applying
\refe{gamma-upper-generic} and observing that, by \refe{BCH-n} and
\refe{BCH-mlarge},
  $$ \frac{dr}{m\sqrt n} = \frac nm \frac{dr}{n^{3/2}}
                                           < 3n \frac{\log_2(dr)}m \le 12n, $$
whence
\begin{align*}
  \lpr 1-\frac\mu{n}\rpr r +n+d-\mu
       &<   \lpr 1-\frac\mu{n}\rpr r + 2n \\
       &\le \lpr 0.5+\frac{e-1}{\sqrt n}\rpr r + 2n \\
       &< 0.5 r + \lpr \frac{dr}{m\sqrt n} + 2n \rpr \\
       &<   0.5r + 14n.
\end{align*}
\end{proof}

\begin{proof}[Proof of Theorem \reft{beta-lower}]
Observing that
  $$ (2d-\rho)\lfl\frac{r}{2d}\rfl + \rho \lpr\lfl\frac{r}{2d}\rfl+1\rpr
                              = \lfl\frac{r}{2d}\rfl \cdot 2d + \rho = r, $$
choose subspaces $V_1\longc V_{2d}\le\Fr$ with $\Fr=V_1\oplus\dotsb\oplus
V_{2d}$ so that
  $$ \dim V_1 \longe \dim V_{2d-\rho} = \lfl\frac{r}{2d}\rfl $$
and
  $$ \dim V_{2d-\rho+1} \longe \dim V_{2d} = \lfl\frac{r}{2d}\rfl + 1. $$
In every subspace $V_i$ fix a basis $\fe_i$. For $v\in\Fr$ let $\supp v$
denote the support of $v$ with respect to the union of the bases $\fe_i$, and
for each $i\in[1,2d]$ let $\supp_i(v):=\supp(v)\cap\fe_i$. Also, let
$w(v)=|\supp(v)|$ and $w_i(v):=\supp_i(v)$; that is, $w(v)$ is the weight of
$v$ with respect to the union of the bases $\fe_i$, and $w_i(v)$ is the
contribution of $\fe_i$ to $w(v)$, so that $w=w_1\longp w_{2d}$. Finally, set
  $$ B := \{ v\in\Fr
        \colon w(v)=d\ \text{ and } w_1(v)\longc w_{2d}(v) \le 1 \}; $$
thus,
  $$ |B| = \sum_\sub{0\le i\le 2d-\rho,\,0\le j\le\rho \\ i+j=d}
          \binom{2d-\rho}{i} \binom{\rho}{j}
                 \lfl\frac{r}{2d}\rfl^i \lpr\lfl\frac{r}{2d}\rfl+1\rpr^j, $$
and we show that through every $v\in\Fr$ passes a flat $F_v$ of co-dimension
at most $d$, disjoint with $B\stm\{v\}$. We distinguish three cases.

If $v\in B$, then we let
  $$ F_v := \{ u\in\Fr\colon \supp v\seq\supp u \}. $$
Evidently, we have $v\in F_v$ and $\codim F_v=w(v)=d$; moreover, if
 $u\in F_v\stm\{v\}$, then $w(u)>w(v)=d$, implying $u\notin B$.

If there exists $i\in[1,2d]$ with $w_i(v)\ge 2$, then we choose
$E\seq\supp_i(v)$ with $|E|=2$ and set
  $$ F_v := \{ u\in\Fr\colon E\seq\supp(u) \}. $$
We have $v\in F_v,\ F_v\cap B=\est$, and $\codim F_v=2$.

Finally, if $v\notin B$ and $w_i(v)\le 1$ for each $i\in[1,2d]$, then there
exists $I\seq[1,2d]$ with $|I|=d+1$ such that for all $i\in I$, the weights
$w_i(v)$ are equal to each other. In this case we take $F_v$ to be the
co-$d$-flat (actually, a co-$d$-subspace) consisting of those $u\in\Fr$ with
the property that for all $i\in I$, the weights $w_i(u)$ are of the same
parity. It is immediately verified that $v\in F_v$ and $F_v\cap B=\est$.
\end{proof}

\begin{proof}[Proof of Theorem \reft{beta-lowerprime}]
Choose subspaces $V_1\longc V_d\le\F_r$ with
 $\Fr=V_1\oplus\dotsb\oplus V_d$ so that
  $$ \dim V_1\longe\dim V_{d-\rho} = \lfl\frac{r}{d}\rfl $$
and
  $$ \dim V_{d-\rho+1}\longe\dim V_d = \lfl\frac{r}{d}\rfl + 1. $$
In every subspace $V_i$ fix a basis $\fe_i$, and define the sets
 $\supp,\ \supp_i$, and the weight functions $w$ and $w_i$ as in the proof of
Theorem~\reft{beta-lower}. Let
  $$ B := \{ v\in\Fr\colon w_i(v)\le 1,\ i\in[1,d] \}; $$
thus,
  $$ |B| = \lpr \lfl\frac{r}{d}\rfl+1 \rpr^{d-\rho}
                                    \lpr \lfl\frac{r}{d}\rfl+2\rpr^{\rho}, $$
and we claim that through every $v\in\Fr$ passes a flat $F_v$ of co-dimension
at most $d$, disjoint with $B\stm\{v\}$. To show this we distinguish two
cases, according to whether $v\in B$ or $v\notin B$.

If there exists $i\in[1,d]$ with $w_i(v)>1$ (that is, $v\notin B$), then we
choose $E\seq\supp_i(v)$ with $|E|=2$, and set
  $$ F_v := \{ u\in\Fr\colon E\seq\supp(u) \}. $$
Clearly, this is a flat of co-dimension $2$, disjoint with $B$ and passing
through $v$.

If, on the other hand, we have $w_i(v)\le 1$ for each $i\in[1,d]$ (that is,
$v\in B$), then we consider the partition $[1,d]=I_0\cup I_1$ with
  $$ I_\nu := \{ i\in[1,d]\colon w_i(v)=\nu \};\quad \nu\in\{0,1\} $$
and let $F_v$ be the co-$d$-flat consisting of those vectors $u\in\Fr$ with
the property that $w_i(u)$ is even for each $i\in I_0$, and
 $\supp_i v\seq\supp_i u$ for each $i\in I_1$. It is immediately verified
that $F_v\cap B=\{v\}$.
\end{proof}

\begin{proof}[Proof of Theorem \reft{beta-lower2}]
Letting $K:=r-(r_1\longp r_k),\ r_{k+1}\longe r_{k+K}=1$, and
 $d_{k+1}\longe d_{k+K}=0$, we see that $r_1\longp r_k=r$ can be assumed without
loss of generality. With this extra assumption, we choose subspaces
$V_1\longc V_k\le\Fr$ so that $\Fr=V_1\oplus\dotsb\oplus V_k$ and $\dim
V_i=r_i$ for $i=1\longc k$, in every subspace $V_i$ fix a basis $\fe_i$, and
define $\supp,\ \supp_i,\ w$, and $w_i$ as in the proofs of
Theorems~\reft{beta-lower} and~\reft{beta-lowerprime}. Finally, we set
  $$ B := \{ v\in\Fr \colon w_i(v)=d_i,\ i\in[1,k] \}; $$
thus,
  $$ |B| = \binom{r_1}{d_1}\dotsb\binom{r_k}{d_k}, $$
and we claim that through every $v\in\Fr$ passes a flat $F_v$ of co-dimension
at most $d$, disjoint with $B\stm\{v\}$. To show this we distinguish several
cases: the case where $v\in B$, that where
 $w_i(v)\ge d_i+1$ for some $i\in[1,k]$, and that where $w_i(v)\le d_i-1$
for some $i\in[1,k]$, with the last two cases further splitting into two
subcases each.

If $v\in B$, then we take $F_v:=\{u\in\Fr\colon \supp v\seq\supp u\}$.
Clearly, $F_v\cap B=\{v\}$, and the co-dimension of $F_v$ is
 $d_1\longp d_k\le d$.

If there exists $i\in[1,k]$ with $w_i(v)\ge d_i+1$, then we find
$E\seq\supp_i(v)$ with $|E|=d_i+1$, and set
\begin{align*}
  F_v &:= \{ u\in\Fr \colon E\seq\supp(u) \}\quad \text{if}\ d_i<d,
\intertext{and}
  F_v &:= \{ u\in\Fr \colon E\seq\supp(u)
                   \ \text{or}\ E\cap\supp(u)=\est \}\quad \text{if}\ d_i=d.
\end{align*}
Clearly, we have $v\in F_v$, and
  $$ \codim F_v = \begin{cases}
       |E| = d_i+1 \le d &\ \text{if}\ d_i<d, \\
       |E|-1=d_i=d &\ \text{if}\ d_i=d.
                  \end{cases} $$
Furthermore, we have $F_v\cap B=\est$: for, if $E\seq\supp(u)$, then
$w_i(u)\ge|E|>d_i$, and if $E\cap\supp(u)=\est$ and $d_i=d$, then
  $$ w_i(u)\le r_i-|E| \le (d_i+d)-(d_i+1) < d = d_i. $$

In a similar way we treat the situation where $w_i(v)\le d_i-1$ for some
$i\in[1,k]$. In this case we find a set $E\seq\fe_i\stm\supp_i(v)$ with
$|E|=r_i-d_i+1$, and let
\begin{align*}
  F_v &:= \{ u\in\Fr \colon E\cap\supp(u)=\est \}\quad \text{if}\ r_i<d_i+d,
\intertext{and}
  F_v &:= \{ u\in\Fr \colon E\cap\supp(u)=\est
                \ \text{or}\ E\seq\supp(u) \}\quad \text{if}\ r_i=d_i+d.
\end{align*}
Thus, $v\in F_v$, the co-dimension of $F_v$ is
  $$ \codim F_v = \begin{cases}
       |E| = r_i-d_i+1 \le d &\ \text{if}\ r_i<d_i+d, \\
       |E|-1=r_i-d_i=d &\ \text{if}\ r_i=d_i+d,
                  \end{cases} $$
and $F_v$ is disjoint with $B$: for, if $\supp(u)\cap E=\est$, then
$w_i(u)\le r_i-|E|=d_i-1<d_i$, and if $E\seq\supp(u)$ and $r_i=d_i+d$, then
$w_i(u)\ge|E|=d+1>d_i$, implying $u\notin B$ in both cases.
\end{proof}

\begin{proof}[Proof of Corollary \refc{rk}]
Let
  $$  k:=\lfl\frac{r}{d}\rfl,\ d_1:=\lfl\frac{d} k\rfl,
                             \ \text{and}\ \ r_1:=\lfl\frac{d_1}{d}\,r\rfl; $$
thus,
\begin{equation}\label{e:kappabounds}
  2\le  k < d
\end{equation}
(as $2\le\frac rd<d$),
and
\begin{equation}\label{e:d1r1large}
  \frac dr \le \frac{d_1}{r_1} \le \frac12
\end{equation}
(the first inequality following from $r_1\le\frac{d_1}{d}\,r$, the second
from $\frac{d_1}{d}\,r\ge 2d_1$). Observing also that
  $$ kd_1\le d,\ kr_1 \le kd_1\frac rd\le r, $$
and
  $$ r_1-d_1 = \lfl \lpr \frac{r}{d}\,-1\rpr d_1\rfl \le  k d_1 \le d, $$
we apply Theorem~\reft{beta-lower2} with $r_2\longe r_k=r_1$ and
 $d_2\longe d_k=d_1$ to get
  $$ \bet_r(d) \ge \binom{r_1}{d_1}^ k. $$
Consequently, \refe{binomial} yields
\begin{equation}\label{e:betmp1}
  \ln \bet_r(d) \ge k r_1 H(d_1/r_1)-( k/2)\ln(2r_1).
\end{equation}

Now from
\begin{multline*}
  r_1+ k \ge \lfl \frac rd+\frac{d_1}d\,r\rfl - 1
                        = \lfl(d_1+1) \frac rd \rfl - 1
                             \ge \lfl \frac{d+1} k\,\frac rd \rfl -1 \\
  = \lfl \frac r k + \frac r{ k d} \rfl - 1
      > \frac r k + \frac r{ k d} - 2 \ge \frac r k - 1
\end{multline*}
and \refe{kappabounds} we deduce
 $$  k r_1 > r- k^2- k \ge r-(3/2) k^2, $$
from \refe{d1r1large} and the fact that $H$ is increasing on $[0,1/2]$ we
conclude that
  $$ H(d_1/r_1) \ge H(d/r), $$
and $r_1\le \frac{d_1}d\,r\le\frac r k$ along with \refe{kappabounds} gives
$2r_1\le r$. Combining these observations with \refe{betmp1} we obtain
  $$ \ln \bet_r(d) > (r-(3/2) k^2)H(d/r)-( k/2)\ln r. $$

To derive the first assertion of the corollary we now notice that the
inequality
  $$ H(t) \le t\ln(e/t),\ t\in[0,1] $$
gives
  $$ (3/2)  k^2H(d/r)+( k/2)\ln r \le \frac32\, \frac rd \ln\frac{er}d
                                + \frac12\, \frac rd\ln r < 2\frac rd\ln r $$
since $d\ge 3$ by \refe{kappabounds}. For the second assertion just observe
that if $d/\sqrt r\to\infty$, then $rH(d/r)\ge d\ln\frac rd$ whereas
 $\frac rd\ln r=o(d\ln\frac rd)$, and use \refe{binomial}.
\end{proof}

\bigskip

\end{document}